\documentclass{amsart}

\usepackage{amssymb,amsmath}
\usepackage{amsthm}
\usepackage{graphicx}
\usepackage{tikz}
\usepackage{fancyhdr}
\usepackage[matrix,arrow,tips]{xy}
\usepackage{array,textcomp}

\usepackage{float}

\newcommand{\eps}{\epsilon}

\newcommand{\bbC}{\mathbb{C}}
\newcommand{\bbD}{\mathbb{D}}
\newcommand{\bbZ}{\mathbb{Z}}
\newcommand{\bbN}{\mathbb{N}}
\newcommand{\bbR}{\mathbb{R}}

\newcommand{\bbM}{\mathbb{M}}

\newcommand{\tr}{\mathrm{tr}}
\newcommand{\End}{\mathrm{End}}

\newcommand{\Hom}{\mathrm{Hom}}

\newcommand{\lag}{\mathfrak{g}}
\newcommand{\lah}{\mathfrak{h}}

\newcommand{\laq}{\mathfrak{q}}

\newcommand{\laso}{\mathfrak{so}}

\newcommand{\GL}{\mathrm{GL}}
\newcommand{\SL}{\mathrm{SL}}
\newcommand{\SO}{\mathrm{SO}}
\newcommand{\Sp}{\mathrm{Sp}}
\newcommand{\Spin}{\mathrm{Spin}}
\newcommand{\SU}{\mathrm{SU}}
\newcommand{\U}{\mathrm{U}}

\newcommand{\rank}{\mathrm{rank}}

\newcommand{\ind}{\mathrm{ind}}
\newcommand{\pr}{\mathrm{pr}}

\newcommand{\G}{\mathrm{G}}

\theoremstyle{plain}
\newtheorem{theorem}{Theorem}[section]
\newtheorem{lemma}[theorem]{Lemma}
\newtheorem{proposition}[theorem]{Proposition}
\newtheorem{corollary}[theorem]{Corollary}

\newtheorem{definition}[theorem]{Definition}

\theoremstyle{definition}
\newtheorem{example}[theorem]{Example}
\newtheorem{remark}[theorem]{Remark}

\begin{document}

\title[MF--induced representations and orthogonal polynomials]{Multiplicity free induced representations and orthogonal polynomials}

\author{Maarten van Pruijssen}

\address{Universit{\"a}t Paderborn\\
Institut f{\"u}r Mathematik\\
Warburger Str. 100\\
D-33098 Paderborn }

\email{vanpruijssen@math.upb.de}



\maketitle

\begin{center}
\textit{Dedicated to Gert Heckman for his 60th birthday.}
\end{center}

\begin{abstract}
Let $(G,H)$ be a reductive spherical pair and $P\subset H$ a parabolic subgroup such that $(G,P)$ is spherical. The triples $(G,H,P)$ with this property are called multiplicity free systems and they are classified in this paper. Denote by $\pi^{H}_{\mu}=\ind_{P}^{H}\mu$ the Borel-Weil realization of the irreducible $H$-representation of highest weight $\mu\in P^{+}_{H}$ and consider the induced representation $\ind_{P}^{G}\chi_{\mu}=\ind_{H}^{G}\pi^{H}_{\mu}$, a multiplicity free induced representation. Some properties of the spectrum of the multiplicity free induced representations are discussed. For three multiplicity free systems the spectra are calculated explicitly. The spectra give rise to families of multi-variable orthogonal polynomials which generalize families of Jacobi polynomials: they are simultaneous eigenfunctions of a commutative algebra of differential operators, they satisfy recurrence relations and they are orthogonal with respect to integrating against a matrix weight on a compact subset. We discuss some difficulties in describing the theory for these families of polynomials in the generality of the classification.
\end{abstract}

\tableofcontents

\newpage

%----------------------------------------------------------------------------------------------------------------------------------------------------------
%                                                                                     New Section
%----------------------------------------------------------------------------------------------------------------------------------------------------------

\section{Introduction}

Multiplicity free representations of Lie groups are closely connected to special functions. One of the reasons is that as a general rule, multiplicity freeness implies commutativity on various levels. As an example we mention the Jacobi polynomials with geometric parameters which can be obtained from matrix coefficients of a Lie group $G$ that are invariant by translations over a symmetric subgroup $K$. The convolution algebra of $K$-biinvariant functions is commutative and spanned by the Jacobi polynomials. On the level of Lie algebras, the multiplicity free occurrence of the trivial representation in the restriction of irreducible $G$-representations is reflected in the fact that the algebra of $K$-invariant differential operators $U(\lag)^{K}$ admits a commutative quotient. The Harish-Chandra isomorphism transforms operators in this quotient into differential operators of hypergeometric type for functions on a Euclidean space of dimension $\mathrm{rk}(G/K)$. The Jacobi polynomials are simultaneous eigenfunctions for these differential operators. Another property of the Jacobi polynomials is that they satisfy recurrence relations.

Spherical pairs (Definition \ref{def: spherical pairs}) have similar multiplicity free properties and it is thus natural to ask similar questions about the harmonic analysis on spherical spaces. In this paper we use algebraic methods to study some of these questions. Another method to attack problems concerning multiplicity free representations is the theory of visible actions and propagation of multiplicity free representations of Kobayashi \cite{Kobayashi1, Kobayashi2}. These may be particularly helpful when one wants to perform similar analysis on the non-compact Cartan duals of the Lie groups studied in Sections \ref{section: examples} and \ref{section: structure}. 

We consider triples of groups with multiplicity free properties which give rise to matrix valued special functions. For the compact symmetric pair $(\SU(3),\U(2))$ of rank one these functions appear already in \cite{GPT}, in a quest for families of matrix valued orthogonal polynomials that have a Sturm-Liouville property: they are simultaneous eigenfunctions for a second order differential operator. This idea is pushed further in \cite{Koelink--van Pruijssen--Roman1,Koelink--van Pruijssen--Roman2} based on ideas from \cite{Koornwinder}, to end in a general construction of families of matrix valued orthogonal polynomials in \cite{Heckman van Pruijssen} for spherical pairs of rank one. Moreover, the polynomials are simultaneous eigenfunctions for a commutative algebra of differential operators, another commutative quotient of $U(\lag)^{K}$. The essential ingredient for the construction in \cite{Heckman van Pruijssen} is multiplicity free induction from $H$ to $G$ for a spherical pair $(G,H)$. Besides that, the spectrum of such a $G$-representation must admit a suitable partial ordering. 

In this paper we classify the data that is needed for an extension of this theory. We calculate the spectra of three examples and show that they admit a suitable partial ordering. Furthermore, we point out difficulties, mostly in the structure theory for spherical pairs, for a complete generalization of the matrix valued polynomials. We work on the level of algebraic groups defined over $\bbC$, except in the last two sections, where we also consider compact real forms of reductive algebraic groups.

\begin{definition} A $G$-variety $X$ is called spherical if it is normal and if it admits an open orbit for the action of a Borel subgroup $B\subset G$. A pair $(G,F)$ consisting of a reductive group $G$ and a closed algebraic subgroup $F$ is called spherical if $G/F$ is a spherical $G$-variety. 
\end{definition}

\begin{definition}\label{def: spherical pairs}
A triple $(G,H,P)$ is called a multiplicity free system (MFS) if $G$ is connected, reductive, $H\subset G$ is connected, reductive and $P\subset H$ is a parabolic subgroup such that $(G,P)$ is spherical.
\end{definition}

The notion of a MFS depends only on the Lie algebras, so for the classification it is sufficient to look at the indecomposable ones, i.e.~those not of the form $(G_{1}\times G_{2}, H_{1}\times H_{2}, P_{1}\times P_{2})$ where $(G_{i},H_{i},P_{i})$ are MFSs. Moreover, we assume all groups to be connected. The group $G$ may be assumed to be semi-simple because a possible center is always contained in any Borel subgroup. The definition of a MFS implies that $(G,H)$ is a spherical pair and these have been classified by Kr{\"a}mer \cite{Kramer} and Brion \cite{Brion1}. Hence the list of candidates is short.

We need not be concerned with the spherical pairs that are symmetric, as those MFSs have been classified in \cite{He et al}. The rank one cases were classified in \cite{Heckman van Pruijssen}. The list of MFSs $(G,H,P)$ with $(G,H)$ non-symmetric and $P$ non-trivial, i.e.~$P\ne H$, turns out to be fairly small: In Section \ref{section: MFS} we find 11 examples among which there are 8 families. 

Given a MFS $(G,H,P)$ and a character $\mu:P\to\bbC^{\times}$, the induced representation $\pi^{H}_{\mu}=\ind_{P}^{H}\mu$ is irreducible by the Borel-Weil Theorem. The induction of $\pi_{\mu}^{H}$ to $G$ decomposes multiplicity free by Frobenius Reciprocety and the spectrum $P^{+}_{G}(\mu)$ of $\ind^{G}_{H}\pi^{H}_{\mu}$ is important in view of the construction of the polynomials. We discuss the some quantitative properties of $P^{+}_{G}(\mu)$ in Section \ref{section: spectrum} and we provide a result concerning the stability of multiplicities.

An explicit description of the spectra $P^{+}_{G}(\mu)$ is obtained in Section \ref{section: examples} for the three MFSs $(G,H,P)$ where $P$ is non-trivial and where $(G,H)$ is $(\Spin_{9},\Spin_{7})$, $(H\times H,H)$ with $H=\SL_{n+1}$, or $(\Sp_{2m}\times\Sp_{2n},\Sp_{2m-2}\times\Sp_{2}\times\Sp_{2n-2})$.

The spectra that are known behave well with respect to the decomposition of the tensor product with fundamental spherical representations. This leads to a theory of families of multi-variable matrix valued orthogonal polynomials, which is discussed in Section \ref{section: OP}. In Section \ref{section: structure} we discuss some difficulties for generalizing this construction to the MSFs in the classifications.

The following notations and conventions are employed in this paper: Groups are indicated with Latin capitals, their Lie algebras with their gothic counterparts. The roots and weights that occur are numbered as in \cite[App.~C]{KnappLGBI}. The weight semi-group of an algebraic group $G$ is denoted by $P^{+}_{G}$. Given a weight $\lambda\in P^{+}_{G}$, an irreducible representation of $G$ of highest weight $\lambda$ is denoted by $\pi^{G}_{\lambda}$ and its representation space by $V^{G}_{\lambda}$. The restriction $\pi^{G}_{\lambda}|_{H}$ of the irreducible representation $\pi^{G}_{\lambda}$ to a reductive subgroup $H\subset G$ decomposes into irreducible $H$-representations $\pi^{H}_{\mu}$. Their multiplicities are denoted by $m^{G,H}_{\lambda}(\mu)$. The dual vector space of a vector space $V$ is denoted by $V^{\vee}$. The weight of the irreducible representation contragredient to $\pi^{G}_{\lambda}$ is denoted by $\lambda^{\vee}$.

%----------------------------------------------------------------------------------------------------------------------------------------------------------
%                                                                                     New Section
%----------------------------------------------------------------------------------------------------------------------------------------------------------

\section{Multiplicity free systems}\label{section: MFS}

The following result has been established in e.g. \cite[Lemma 5.2]{Knop van Steirteghem} and \cite[\S2.1]{Panyushev}.

\begin{lemma}
Let $B\subset G$ be a Borel subgroup such that $BH\subset G$ is open. Then $B\cap H$ is a Borel subgroup of a generic isotropy group $H_{*}\subset H$ for $H$ acting on $\lah^{\perp}\subset\lag^{\vee}$.
\end{lemma}

This result provides a criterion for a triple $(G,H,P)$ to be a MFS. 

\begin{proposition}\label{prop: fiber reduction}
Let $G$ be a reductive group and $H\subset G$ a reductive spherical subgroup. Let $P\subset H$ be a parabolic subgroup and let $H_{*}$ be a generic isotropy group for $H$ acting on $\lah^{\perp}$. The pair $(G,P)$ is spherical if and only if $H/P$ is $H_{*}$-spherical.
\end{proposition}

\begin{proof}
Let $B\subset G$ be a Borel subgroup such that $BH\subset G$ is open. Consider the map $c:G/P\to G/H$. Since $c^{-1}(BH/H)$ is open and $B$-stable, $G/P$ has an open $B$-orbit if and only if $c^{-1}(BH/H)$ has an open $B$-orbit. The latter holds if and only if $B\cap H$ has an open orbit in the fiber $c^{-1}(H/H)=H/P$ since $BH/H\cong B/(B\cap H)$. This is equivalent to $H/P$ being $H_{*}$-spherical.   
\end{proof}

Let $X$ be a $G$-variety. The complexity $c(X)$ is the codimension of a generic $B$-orbit. The weight lattice $\Lambda(X)$ is the set of weights of all rational $B$-eigenfunctions and its rank is called the (spherical) rank of $X$. The complexity and rank of $G/H$ are related to the rank and dimension of $G$ and $H_{*}$ according to the formulas:
\begin{eqnarray}
2c(G/H)+r(G/H)=\dim G-2\dim H+\dim H_{*},\label{formulaA}\\
r(G/H)=\rank(G)-\rank(H_{*}),\label{formulaB}
\end{eqnarray}
see e.g.~\cite[Ch.~9]{Timashev}. A necessary condition for $H/P$ to be $H_{*}$-spherical is $\dim B_{H_{*}}+\dim P\ge\dim H$. With Proposition \ref{prop: fiber reduction} and formulas (\ref{formulaA}) and (\ref{formulaB}), this implies the following result.

\begin{corollary}\label{cor: criterion}
A necessary condition for $(G,H,P)$ to be a MFS is $\dim P\ge|R^{+}_{G}|$.
\end{corollary}

This is in fact the ordinary dimension condition $\dim B+\dim P\ge\dim G$. The strategy to obtain a classification of MFSs is going down the list of reductive spherical pairs that was obtained by Brion \cite{Brion1}, restrict to all the non-symmetric examples $(G,H)$ (collected in \cite[Tables 10.1, 10.3]{Timashev}) and then check for all the parabolic subgroups $P\subset H$ whether they are spherical in $G$.

\begin{table}[h!]
\begin{center}
 \begin{tabular}{|c|c|c|c|c|}
  \hline
 \textnumero & $G$ & $H$  & $H_{*}$ & $J^{c}_{H}$\\
  \hline
 1a & $\SL_{n+m}$  & $\SL_{m}\times\SL_{n}$  & $\SL_{m-n}\times S((\bbC^{\times})^{n})$ & $\emptyset$\\
&&$m> n\ge3$&&\\
    \hline
 1b & $\SL_{n+2}$  & $\SL_{n}\times\SL_{2}$ & $\SL_{n-2}\times S((\bbC^{\times})^{2})$ & $\{\alpha_{n+1}\}$ \\
&&$n\ge3$&& \\
   \hline
 1c & $\SL_{m+1}$  & $\SL_{m},m\ge2$ & $\SL_{m-1}$ & $\Pi_{H}\backslash\{\alpha_{i}\}$ \\
    \hline

2 & $\SL_{2n+1}$  & $\Sp_{2n}\times\bbC^{\times}$ & $\bbC^{\times}$ & $\emptyset$\\
    \hline
3 & $\SL_{2n+1}$  & $\Sp_{2n}$ & $\{e\}$  & $\emptyset$\\
    \hline
4 & $\Sp_{2n}$  & $\Sp_{2n-2}\times\bbC^{\times}$ & $\Sp_{2n-4}$  & $\{\alpha_{n-1}\}$\\
    \hline
5 & $\SO_{2n+1}$  & $\GL_{n}$ & $\{e\}$ & $\emptyset$\\
    \hline
6 & $\SO_{4n+2}$  & $\SL_{2n+1}$ &  $(\SL_{2})^{n}$& $\emptyset$\\
    \hline
7 & $\SO_{10}$  & $\Spin_{7}\times\SO_{2}$ & $\SL_{2}$  & $\emptyset$\\
    \hline
8 & $\SO_{9}$  & $\Spin_{7}$ & $\SL_{3}$ & $\{\alpha_{1}\}$\\
    \hline
9 & $\SO_{8}$  & $\G_{2}$ & $\SL_{2}$ & $\emptyset$\\
    \hline
10 & $\SO_{7}$  & $\G_{2}$ & $\SL_{3}$ & $\{\alpha_{1}\},\{\alpha_{2}\}$\\
    \hline
11 & $\mathrm{E}_{6}$  & $\Spin_{10}$ & $\SL_{4}$  & $\emptyset$\\
    \hline
12 & $\G_{2}$  & $\SL_{3}$ &  $\SL_{2}$ & $\{\alpha_{1}\},\{\alpha_{2}\}$\\
    \hline

\end{tabular}
  \vspace{0.2cm}
\caption{MFSs where $(G,H)$ spherical, non-symmetric and $G$ simple. For nos.~1a,b and 4 the roots of $H$, the semi-simple part of a maximal non-trivial Levi subgroup, are identified with the roots of $G$.}\label{table1}
\end{center}
\end{table}

\begin{remark}\label{remark: tables}
The embeddings of the spherical subgroups $H\subset G$ in Tables \ref{table1}, \ref{table2} are the standard ones according to the literature \cite{Brion1, Kramer, Timashev}. The subgroups $H_{*}\subset H$ are indicated up to isogeny. For most cases the groups have been determined already in \cite{Knop van Steirteghem}. The parabolic subgroups of $H$ are given by subgroups $J_{H}\subset\Pi_{H}$. As the parabolic subgroups are big, the complements $J_{H}^{c}$ are displayed. Subsets of $J_{H}^{c}$ also give parabolic subgroups of $H$ that are spherical in $G$. These are not indicated.
\end{remark}

\begin{table}[h!]
\begin{center}
 \begin{tabular}{|m{5mm}|m{4.5cm}|m{10mm}|}
  \hline
\textnumero & \hspace{10mm}$H_{*}\subset H\subset G$ & $J^{c}_{H}$ \\
\hline
\hspace{1mm}1&
\begin{tikzpicture}[scale=0.2]
\begin{scope}
%1
\clip (1,3) rectangle (24,8);

\draw[black] (4,6) -- (4,4) -- (6,6) -- (6,4);

\draw[fill=white] (4,6) circle (15pt);
\draw[fill=white] (6,6) circle (15pt);

\draw[fill=black] (4,4) circle (14pt);
\draw[fill=black] (6,4) circle (14pt);

%\draw[fill=lightgray] (4,2) circle (15pt);

\draw (16,6.2) node {\(A_{n-1}A_{n}\)};
\draw (16,4) node {\(A_{n-1}T\)};

\end{scope}
\end{tikzpicture}
&\hspace{2mm}$\emptyset$\\
\hline
\hspace{1mm}2&
\begin{tikzpicture}[scale=0.2]
\begin{scope}
%2
\clip (1,1) rectangle (24,8);

\draw[black] (4,2) -- (4,6) -- (6,4) -- (6,6);

\draw[fill=white] (4,6) circle (15pt);
\draw[fill=white] (6,6) circle (15pt);

\draw[fill=black] (4,4) circle (14pt);
\draw[fill=black] (6,4) circle (14pt);

\draw[fill=lightgray] (4,2) circle (15pt);

\draw (16,6.2) node {\(C_{n}C_{2}\)};
\draw (16,4) node {\(C_{n-2}C_{2}\)};
\draw (16,1.8) node {\(C_{n-4}\)};

\end{scope}
\end{tikzpicture}
&\hspace{2mm}$\emptyset$\\ 
\hline
\hspace{1mm}3&
\begin{tikzpicture}[scale=0.2]
\begin{scope}
%3
\clip (1,1) rectangle (24,8);

\draw[black] (3,2) -- (3,4) -- (4,6) -- (5,4) -- (6,6) -- (7,4) -- (7,2);
%\draw[black] (5,4) -- (5,2);

\draw[fill=white] (4,6) circle (15pt);
\draw[fill=white] (6,6) circle (15pt);

\draw[fill=black] (3,4) circle (14pt);
\draw[fill=black] (5,4) circle (14pt);
\draw[fill=black] (7,4) circle (14pt);

\draw[fill=lightgray] (3,2) circle (15pt);
%\draw[fill=lightgray] (5,2) circle (15pt);
\draw[fill=lightgray] (7,2) circle (15pt);

\draw (16,6.2) node {\(A_{n-1}C_{m}\)};
\draw (16,4) node {\((TA_{n-3})A_{1}C_{m-1}\)};
\draw (16,1.8) node {\((TA_{n-5})C_{m-2}\)};

\end{scope}

\end{tikzpicture}
&\hspace{1mm}$\{\beta\}$\\
\hline
\hspace{1mm}4&
\begin{tikzpicture}[scale=0.2]
\begin{scope}
%4
\clip (1,3) rectangle (24,8);

\draw[black] (4,6) -- (5,4) -- (6,6);

\draw[fill=white] (4,6) circle (15pt);
\draw[fill=white] (6,6) circle (15pt);

\draw[fill=black] (5,4) circle (14pt);

\draw (16,6.2) node {\(B_{n}D_{n}\)};
\draw (16,4) node {\(B_{n}\vee D_{n}\)};

\end{scope}

\end{tikzpicture}
&\hspace{2mm}$\emptyset$ \\ 
\hline
\hspace{1mm}5&
\begin{tikzpicture}[scale=0.2]
\begin{scope}
%5
\clip (1,1) rectangle (24,8);

\draw[black] (3,2) -- (3,4) -- (4,6) -- (5,4) -- (6,6) -- (7,4) -- (7,2);
%\draw[black] (5,4) -- (5,2);

\draw[fill=white] (4,6) circle (15pt);
\draw[fill=white] (6,6) circle (15pt);

\draw[fill=black] (3,4) circle (14pt);
\draw[fill=black] (5,4) circle (14pt);
\draw[fill=black] (7,4) circle (14pt);

\draw[fill=lightgray] (3,2) circle (15pt);
%\draw[fill=lightgray] (5,2) circle (15pt);
\draw[fill=lightgray] (7,2) circle (15pt);

\draw (16,6.2) node {\(A_{n-1}C_{m}\)};
\draw (16,4) node {\(A_{n-3}A_{1}C_{m-1}\)};
\draw (16,1.8) node {\((TA_{n-5})C_{m-2}\)};

\end{scope}

\end{tikzpicture}
&\hspace{1mm}$\{\beta\}$\\
\hline
\hspace{1mm}6&
\begin{tikzpicture}[scale=0.2]
\begin{scope}
%6
\clip (1,1) rectangle (24,8);

%\draw[dotted] (1,1) -- (1,7) -- (9,7) -- (9,1) --cycle;
%\draw[dotted] (.1,.1) -- (.1,7.9) -- (9.9,7.9) -- (9.9,.1) --cycle;
%\draw (1,7) node {\(6\)};

\draw[black] (2,2) -- (2,6) -- (8,4) -- (4,6) -- (4,2);
\draw[black] (8,4) -- (6,6) -- (6,2);

\draw[fill=white] (2,6) circle (15pt);
\draw[fill=white] (4,6) circle (15pt);
\draw[fill=white] (6,6) circle (15pt);

\draw[fill=black] (2,4) circle (14pt);
\draw[fill=black] (4,4) circle (14pt);
\draw[fill=black] (6,4) circle (14pt);
\draw[fill=black] (8,4) circle (14pt);

\draw[fill=lightgray] (2,2) circle (15pt);
\draw[fill=lightgray] (4,2) circle (15pt);
\draw[fill=lightgray] (6,2) circle (15pt);

\draw (16,6.2) node {\(C_{l}C_{m}C_{n}\)};
\draw (16.6,4) node {\(C_{l-1}C_{m-1}C_{n-1}C_{1}\)};
\draw (16,1.8) node {\(C_{l-2}C_{m-2}C_{n-2}\)};

\end{scope}

\end{tikzpicture}
&$\{\beta,\beta'\}$\\ 
\hline
\hspace{1mm}7&

\begin{tikzpicture}[scale=0.2]
\begin{scope}
%7
\clip (1,1) rectangle (24,8);

\draw[black] (3,2) -- (3,4) -- (4,6) -- (5,4) -- (6,6) -- (7,4) -- (7,2);
\draw[black] (5,4) -- (5,2);

\draw[fill=white] (4,6) circle (15pt);
\draw[fill=white] (6,6) circle (15pt);

\draw[fill=black] (3,4) circle (14pt);
\draw[fill=black] (5,4) circle (14pt);
\draw[fill=black] (7,4) circle (14pt);

\draw[fill=lightgray] (3,2) circle (15pt);
\draw[fill=lightgray] (5,2) circle (15pt);
\draw[fill=lightgray] (7,2) circle (15pt);

\draw (16,6.2) node {\(C_{n}C_{m}\)};
\draw (16,4) node {\(C_{n-1}C_{1}C_{m-1}\)};
\draw (16,1.8) node {\(C_{n-2}TC_{m-2}\)};

\end{scope}

\end{tikzpicture}
&\hspace{1mm}$\{\beta\}$\\
\hline
\hspace{1mm}8&
\begin{tikzpicture}[scale=0.2]
\begin{scope}
% 8
\clip (1,1) rectangle (24,8);

\draw[black] (2,2) -- (2,4) -- (3,6) -- (4,4) -- (5,6) -- (6,4) -- (7,6) -- (8,4) -- (8,2);

\draw[fill=white] (3,6) circle (15pt);
\draw[fill=white] (5,6) circle (15pt);
\draw[fill=white] (7,6) circle (15pt);

\draw[fill=black] (2,4) circle (14pt);
\draw[fill=black] (4,4) circle (14pt);
\draw[fill=black] (6,4) circle (14pt);
\draw[fill=black] (8,4) circle (14pt);

\draw[fill=lightgray] (2,2) circle (15pt);
\draw[fill=lightgray] (8,2) circle (15pt);

\draw (16,6.2) node {\(C_{m}C_{2}C_{n}\)};
\draw (16,4) node {\(C_{m-1}C_{1}C_{1}C_{m-1}\)};
\draw (16,1.8) node {\(C_{m-2}C_{n-2}\)};

\end{scope}
\end{tikzpicture}
&\hspace{1mm}$\{\beta,\beta'\}$\\ 

 \hline
 \end{tabular}
  \vspace{0.2cm}
\caption{MFSs where $(G,K)$ spherical, non-symmetric and $G$ not simple. The roots $\beta,\beta'$ are roots of the factors $\SL_{2}$ or $\Sp_{2}$ (without parameter). We have indicated only the Dynkin types, wherre $T$ indicates a torus $\bbC^{\times}$.}\label{table2}
\end{center}
\end{table}

\begin{theorem}\label{thm: classification}
The indecomposable spherical pairs $(G,H,P)$ with $(G,H)$ not symmetric, are classified in Tables 1 and 2 (see Remark \ref{remark: tables} for an explanation of the tables). 
\end{theorem}

\begin{proof}[Proof for the items in Table \ref{table1}]
The spherical pairs $(G,H)$ in nos.~1,6 and 11 satisfy $H=K'=(K,K)$ (commutator subgroup), where $K\subset G$ is symmetric. The generic isotropy groups $H_{*}$ are smaller than the corresponding generic isotropy groups $K_{*}\subset K$, viz.~$K_{*}/H_{*}=\bbC^{\times}$, see e.g.~\cite[Beh.~3.1]{Kramer}. The parabolic subgroups of $H$ and $K$ are in 1--1-correspondence, denoted by $H\supset P'\leftrightarrow P\subset K$. If $H/P'$ is $H_{*}$-spherical, then $K/P$ is $K_{*}$-spherical, because $H/P'\cong K/P$. This leaves only a few parabolic subgroups that we have to check (those from \cite{He et al}). Note that the subgroups $H$ are contained in a Levi subgroup of $G$. This also holds for no.~4. To see that the subgroups $H\subset G$ and possible parabolic subgroups are spherical we invoke \cite[Prop.~I.1]{Brion1}: Let $L\subset P$ be a Levi subgroup and choose a parabolic subgroup $Q_{P}\subset G$ such that $P$ is regularly embedded in $Q_{P}$. The Levi subgroup of $Q_{P}$ is $L$ or $L\times\bbC^{\times}$. The group $P\subset G$ is spherical if and only if (1) $L\subset H$ is spherical and (2) if $B_{H}\subset H$ is a Borel subgroup with $B_{H}L\subset H$ open, then $B_{H}\cap L$ must have an open orbit in $Q_{P,u}/P_{u}$. Condition (1) is always satisfied and it shows that we may take any Borel subgroup to check property (2). The action of $L$ on $Q_{P,u}/P_{u}$ can be linearized to the adjoint action of $L$ on $\laq_{H,u}$. Spherical actions of reductive groups on vector spaces have been classified in \cite{Benson and Ratcliff, Brion2, Kac,Leahy} and a careful check of the tables leads to the indicated parabolic subgroups $P\subset H$. We used the tables in \cite{Knop: some remarks} to deal with the cases where the representations are not saturated.

The generic isotropy group of no.~2 is abelian, so irreducible representations of $H$ cannot decompose multiplicity free if they are of dimension $>1$. Nos.~3 and 5 have $|R^{+}_{G}|=\dim H$, so Corollary \ref{cor: criterion} implies that there are no non trivial parabolic subgroups that give MFSs. We apply the criterion also to no.~7, where $|R^{+}_{G}|=20$ and where the maximal parabolic subgroups are of dimension 15, 16 and 17.

No.~8 has $|R^{+}_{G}|=12$ while the maximal parabolic subgroups of $\G_{2}$ are of dimension $9$. Nos.~10 and 12 are discussed in \cite{Heckman van Pruijssen}. This leaves no.~9, for which we use Proposition \ref{prop: fiber reduction}. There is one candidate for a parabolic subgroup $P$,  it is determined by $\{\alpha_{1}\}^{c}$ and has dimension 16. An irreducible representation of $\Spin_{7}$ of highest weight $k\omega_{1}$ restricted to $H_{*}=\SL_{3}$ decomposes multiplicity free. To see this, note that $H_{*}=\SL_{3}\subset\G_{2}\subset\Spin_{7}=H$. The restriction $\pi^{H}_{k\omega_{1}}|_{\G_{2}}$ remains irreducible and hence $\pi^{H}_{k\omega_{1}}|_{H_{*}}$ decomposes multiplicity free.
\end{proof}

The proof for the items in Table \ref{table2} is postponed to Section \ref{section: spectrum} (below Remark \ref{rk: Sp bottom}) because we use spectra of induced representations for which we have to introduce some notation first. Alternatively, one could prove the Theorem for the items in Table \ref{table2} using Proposition \ref{prop: fiber reduction}. However, for this one needs to know the embeddings $H_{*}\subset H$ which are in general not standard.

%----------------------------------------------------------------------------------------------------------------------------------------------------------
%                                                                                     New Section
%----------------------------------------------------------------------------------------------------------------------------------------------------------

\section{The spectrum}\label{section: spectrum}

Given a multiplicity free system $(G,H,P)$ and an irreducible representation $\pi^{H}_{\mu}$ where $\mu$ is a character of $P$, it is natural to ask which irreducible representations of $G$ contain $\pi^{H}_{\mu}$ upon restriction to $H$. The highest weights of such representations are collected in the set
$$P^{+}_{G}(\mu)=\{\lambda\in P^{+}_{G}:m^{G,H}_{\lambda}(\mu)=1\},$$
baptized as the $\mu$-well. The $0$-well $P^{+}_{G}(0)$ is a monoid generated by a finite number weights, the \textit{fundamental spherical weights}. For $G$ simple the spherical weights are listed in \cite{Kramer}. According to the Borel--Weil Theorem, $\lambda\in P^{+}_{G}(\mu)$ implies $\lambda+\sigma\in P^{+}_{G}(\mu)$ for any $\sigma\in P^{+}_{G}(0)$, because the projection $V^{G}_{\lambda}\otimes V_{\sigma}^{G}\to V_{\lambda+\sigma}^{G}$ is $G$-equivariant and non-trivial. In fact, $m^{G,H}_{\lambda+k\sigma}(\mu)\le m^{G,H}_{\lambda+(k+1)\sigma}(\mu)$ for $\mu\in P_{H}^{+}$ general and all $k\in\bbN$. The weights $\lambda$ for which $m^{G,H}_{\lambda+k\sigma}(\mu)=1$ are determined as follows.  

Consider the parabolic subgroup $Q=\{g\in G:gBH=BH\}$ of $G$ with Levi decomposition $Q_{u}\ltimes L$. The local structure theory implies that there exist an open $B$-stable affine subset $X^{o}\subset X$ and a closed $L$-stable affine subvariety $Z\subset X^{o}$ such that $L$ acts on $Z$ with stabilizers isomorphic to some $L_{0}\subset L$. The horospherical type of $G/H$ is $S=L_{0}\times Q_{u}$. The quotient $A=L/L_{0}$ is a torus because $L'\subset L_{0}$. Moreover, $L_{0}=H_{*}$, the generic isotropy group for $H$ acting on $\lah^{\perp}\subset\lag^{\vee}$, see e.g.~\cite[Thm.~9.1]{Timashev}.

The Levi subgroup $L$ acts irreducibly on $V=(V_{\lambda}^{G})^{Q_{u}}$, and thus so does $L_{0}=H_{*}$, say with highest weight $\lambda_{*}\in P^{+}_{H_{*}}$. As $HB\subset G$ is dense, and $B$ leaves $V$ invariant, any non-zero vector $v\in V$ is $H$-cyclic. It follows that $m^{H,H_{*}}_{\mu}(\lambda_{*})\ge m^{G,H}_{\lambda}(\mu)$ for all $\mu\in P^{+}_{H}$. Define
$$P^{+}_{H_{*}}(\mu)=\{\nu\in P^{+}_{H^{*}}:m^{H,H_{*}}_{\mu}(\nu)\ge1\}.$$
Then the association $\lambda\mapsto\lambda_{*}$ is a mapping $P^{+}_{G}(\mu)\to P^{+}_{H_{*}}(\mu)$. The next result by Kitagawa \cite{Kitagawa} implies that this map is surjective.

\begin{theorem}\label{thm: stability}
Let $\lambda\in P^{+}_{G}$ and let $\sigma\in P^{+}_{G}(0)$ be general, i.e.~$f\in\bbC[G]^{(B)\times H}_{\sigma}$ cuts out the complement of $HB$ in $G$. Then, with the notation from above, $m^{H,H_{*}}_{\mu}(\lambda_{*})=m^{G,H}_{\lambda+k\sigma}(\mu)$ for $m\gg0$.
\end{theorem}

In the case where $(G,H)$ is a symmetric pair, this result was proved by Wallach \cite[Cor.~8.5.15]{Wallach}. Inspired by this we found an algebro-geometric proof.

\begin{proof}[Proof of Theorem \ref{thm: stability}]
Existence of $\sigma$ is guaranteed by the following observation: $HB\subset G$ is affine, being the pre-image of $BH/H=B/(B\cap H)$ of the affine map $G\to G/H$, hence $BH\subset G$ is of codimension one and thus cut out by a regular function $s'$, unique up to multiplication with an invertible element in $\bbC[G]$. Such an element must be a scalar multiple of a character of $G$ and it follows that $s'$ is an eigenfunction of $B\times H$. Hence $s'\in\bbC[G]^{(B)\times(H)}_{(\sigma',\chi')}$. There exists also $0\ne s''\in\bbC[G]^{(B)\times(H)}_{(\sigma'^{\vee},\chi^{-1})}$, whence $s's''\in\bbC[G]^{(B)\times H}$ cuts out the complement of $BH$ in $G$.

Let $L=L_{-\lambda^{\vee}}$ be the line bundle over $G/B$ associated to $\lambda$, i.e.~$\Gamma(G/B,L)\cong V^{G}_{\lambda}$ as $G$-modules. Replace $\sigma$ by a large multiple so that $\lambda-\sigma\not\in P^{+}_{G}$. Define
$$A=\bigoplus_{n\in\bbN}\Gamma(G/P,L_{-\sigma^{\vee}}^{n}),\quad M=\bigoplus_{n\in\bbN}\Gamma(G/B,L\otimes p^{*}L_{-\sigma^{\vee}}^{n}).$$
In view of the isomorphism $\Gamma(G/B,p^{*}L_{-\sigma^{\vee}})=\Gamma(G/P,L_{-\sigma^{\vee}})$ as $G$-modules, the graded $A$-module $M$ defines a coherent sheaf $\widetilde{M}$ over $G/P$. Moreover,
$$\Gamma_{*}(p_{*}L)=\bigoplus_{n\in\bbN}\Gamma(G/P,p_{*}L\otimes L_{-\sigma^{\vee}}^{n})=M$$
by the product formula, whence $p_{*}L=\widetilde{M}$. Let $v_{H}\in V^{G}_{\sigma}$ be a non-trivial $H$-fixed vector and let $D_{+}(v_{H})\subset G/P$ denote the complement of the $H$-invariant divisor on $G/P$. According to the previous observations there is an isomorphism $\Gamma(D_{+}(v_{H}),\widetilde{M})=\Gamma(HB/B,L)$ because $p^{-1}(D_{+}(v_{H}))=HB/B$. Induction in stages implies
$$\Gamma(HB/B,L)=\ind^{H}_{B_{*}}(-\lambda^{\vee})|_{B_{*}}=\ind^{H}_{H_{*}}\pi_{\lambda_{*}}^{H_{*}}.$$
On the other hand, the space $\Gamma(D_{+}(v_{H}),\widetilde{M})$ is the direct limit of the system $V^{G}_{\lambda+k\sigma}\to V^{G}_{\lambda+(k+1)\sigma}$ of $H$-representations, given by the $H$-equivariant map $V^{G}_{\lambda+k\sigma}\to V^{G}_{\lambda+(k+1)\sigma}:v\mapsto \pr(v\otimes v^{H})$ where $v^{H}\in V^{G}_{\sigma}$ is non-zero and $H$-fixed and where $\pr$ is the Cartan projection. This implies that any $H$-isotypical type in $\ind^{H}_{H_{*}}\pi^{H_{*}}_{\lambda_{*}}$ occurs in $V^{G}_{\lambda+k\sigma}$ for $k\gg0$.
\end{proof}

Later, Michel Brion showed me a proof of Theorem \ref{thm: stability}, using invariant theory, which goes along the following lines. Let $B\subset G$ and $B_{H}\subset H$ be Borel subgroups of $G$ and $H$, respectively, for which $BH\subset G$ is open and let $B=TU$ and $B_{H}=T_{H}U_{U}$ be Levi decompositions. Let $U_{H}^{-}$ be opposite to $B_{H}$ in $H$. The torus $T\times T_{H}$ acts on $G/U\times H/U_{H}^{-}$ from the right, and $H$ acts diagonally on this space on the left. We have
$$\bbC[G/U\times H/U_{H}^{-}]^{H\times (T)\times (T_{H})}_{(\lambda,\mu)}=\Hom_{H}(V^{G}_{\lambda},V^{H}_{\mu}).$$
Let $0\ne s\in\bbC[G]^{H\times(B)}_{\sigma}$. Then $\bbC[HB/U]_{(s)}=\bbC[G][s^{-1}]$ is naturally graded. Viewing $s$ as an element as a regular function on $G/U\times H/U_{H}^{-}$ yields, after localizing in the ideal $(s)\subset\bbC[G/U\times H/U_{H}^{-}]$,
$$\bbC[HB/U\times H/U_{H}^{-}]^{H\times (T)\times (T_{H})}_{(\lambda,\mu)}=\bigcup_{n\ge0}\Hom_{H}(V^{G}_{\lambda+n\sigma},V^{H}_{\mu}).$$
On the other hand, $HB/U=HLU/U=HL/U_{L}$, where $U_{L}=B\cap L$.  This can be seen as follows: According to the local structure theory, the multiplication map $R_{u}(Q)\times LH\to HQ$ is an isomorphism. This holds also true for $R_{u}(Q)\times U_{L}\to U$. We have
$$HQ/U=(HL\times R_{u}(Q))/(R_{u}(Q)\times U_{L})=HL\times^{U_{L}}(R_{u}(Q)/R_{u}(Q)),$$
which is isomorphic to $HL/U_{L}$. In turn, $HL/U_{L}=H\times^{H_{*}}(L/U_{L})$, from which we deduce
$$\bbC[HB/U\times H/U_{H}^{-}]^{H\times (T)\times (T_{H})}_{(\lambda,\mu)}=\Hom_{H_{*}}(V^{L}_{\lambda},V^{H}_{\mu}).$$ Since $V^{L}_{\lambda}$ is an irreducible $H_{*}$-module of highest weight $\lambda_{*}=\lambda|_{B_{H_{*}}}$, we find
$$\Hom_{H_{*}}(V_{\lambda_{*}}^{H_{*}},V_{\mu}^{H})= \bigcup_{n\ge0}\Hom_{H}(V^{G}_{\lambda+n\sigma},V^{H}_{\mu}),$$
from which the result follows.

\begin{remark}
It follows that $P^{+}_{G}(\mu)$ has \textit{finite behavior} given by $P^{+}_{H_{*}}(\mu)$ and \textit{asymptotic behavior} given by $P^{+}_{G}(0)$. For general $\mu\in P^{+}_{H}$ there need not be a minimal element in $P^{+}_{G}(\mu)$ over any $\tau\in P^{+}_{H_{*}}(\mu)$, see e.g.~\cite[Remark 3.1]{Camporesi}. In the multiplicity free case that is studied in this paper it may still be the case that such a minimal element exists. In this case the $\mu$-well would have a bottom $B(\mu)\subset P^{+}_{G}(\mu)$ such that $P^{+}_{G}(\mu)=B(\mu)+P^{+}_{G}(0)$. This is the case for the MFSs of rank one \cite{Heckman van Pruijssen} and for the three examples higher rank that are discussed in the next section.
\end{remark}

To finish the proof of Theorem \ref{thm: classification} we need the following result.

\begin{lemma}\label{lemma: Sp spectrum}
Let $(G,H)=(\Sp_{2n},\Sp_{2n-2}\times\Sp_{2})$. Then
$$P^{+}_{G}(\omega_{i}+\ell\omega_{n})\cap P^{+}_{G}(\omega_{i}+(\ell+2)\omega_{n})\ne\emptyset.$$
\end{lemma}

\begin{proof}
Consider the subgroup $H_{1}=\Sp_{2}\times\Sp_{2n-4}\times\Sp_{2}\subset H$ whose embedding is standard, i.e.~the embedding of the first two factors is similar to $H\subset G$ and the third factor is equality. The generic isotropy group $H_{*}=\Sp_{2n-4}\times\Sp_{2}$ is embedded in $H_{1}$ by the diagonal embedding of the factor $\Sp_{2}$. The restriction of $\pi^{H}_{\omega_{i}}$ to $H_{1}$ contains irreducible $\Sp_{2}\times\Sp_{2n-4}$-representations of the shape $\pi^{\Sp_{2}}_{\ell'}\otimes \pi^{\Sp_{2n-4}}_{\eta}$. Restricting further to $H_{*}$ boils down to decomposing $\pi^{\Sp_{2}}_{\ell'}\otimes\pi^{\Sp_{2}}_{\ell(+2)}$ to the diagonal. It follows that $P^{+}_{H_{*}}(\omega_{i}+\ell\omega_{n})\cap P^{+}_{H_{*}}(\omega_{i}+(\ell+2)\omega_{n})\ne\emptyset$ and application of Theorem \ref{thm: stability} yields the result.
\end{proof}

\begin{remark}\label{rk: Sp bottom}
We could have proved Lemma \ref{lemma: Sp spectrum} using the descriptions of the $\mu$-wells from \cite{Heckman van Pruijssen}. For later reference we recall this description for a spacial case: $P^{+}_{G}(\ell\omega_{n})=\ell\varpi_{1}+\bbN\varpi_{2}$, where $(G,H)$ is as in Lemma \ref{lemma: Sp spectrum}.
\end{remark}

\begin{proof}[Proof for the items in Table \ref{table2}] For nos.~1 and 4 we have $|R^{+}_{G}|=\dim H$, so there are no non-trivial MFSs. For the other pairs we reason as follows: either the subgroups $H$ are contained in a Levi subgroup of $G$ (nos.~1, 3 and 5) or they are \textit{very reductive of height 2 or 3}, i.e.~all intermediate groups $H\subset G_{1}\subset G$ are reductive and the longest chains are of length $2$ (nos.~2,4,6,7) and $3$ (nos.~6 and 8), see \cite{Brion1}. The height of $H\subset G$ in no.~6 depends on the parameters. 

We look for parabolic subgroups $P\subset H$ such that the induction from $P$ to $G$ is multiplicity free. In all cases we induce first to an intermediate group $G_{1}$. This induction must be multiplicity free, and the spectrum of the induction cannot contain different representations which, after inducing to $G$, contain the same $G$-representation. The first induction always contains an induction of an $\SL_{2}$-representation to $\SL_{2}\times\SL_{2}$. We invoke Lemma \ref{lemma: Sp spectrum} to exclude all non-trivial parabolic subgroups contained in the factors $\Sp_{2m}$ of $H$. Non-trivial parabolic subgroups $P$ of $\SL_{m-2}$ or $\GL_{m-2}$ in nos.~3 and 5 are excluded by a similar argument. The representations of this group after induction of any $\SL_{2}$-representation is too general for being induced multiplicity free, according to Table \ref{table1} or \cite{He et al}: indeed, the only multiplicity free induced representations come from maximal parabolic subgroups. It follows that the non-trivial parabolic subgroups that we are looking for are precisely those of the factors $\SL_{2},\Sp_{2}$ (not the ones with a parameter).
\end{proof}

\begin{remark}
Part two of the proof of Theorem \ref{thm: classification} indicates how to calculate the $\mu$-wells, one just has to keep track of the wells on different stages. The induction in stages is mostly that of a group-like case or that of a rank one case and both are known. However, for nos.~2, 3 and 5 we need knowledge of the $\mu$-wells for spherical pairs $(G,H)$ of Table \ref{table1}.\end{remark}

\begin{example}
The semi-groups $P^{+}_{G}(0)$ need not be free, as nos.~2 and 7 in Table \ref{table1} show. A similar thing happens for no.~8 in Table \ref{table2}. Inducing the trivial representation of $H$ to $G$ via the intermediate subgroup $C_{m-1}(C_{1}C_{1})(C_{1}C_{1})C_{n-1}$ gives modules of weight $\ell_{1}\varpi_{1}+k_{1}\varpi_{2}+b+\ell_{2}\varpi_{1}'+k_{2}\varpi_{2}'$, where $b\in P^{+}_{\Sp_{4}}(\ell_{1}w_{1}+\ell_{2}w_{2})$, the well for inducing from $C_{1}C_{1}$ to $C_{2}$. It follows that $P^{+}_{G}(0)$ is indeed of rank 6, as expected by (\ref{formulaB}), but it is not free. 
\end{example}

%----------------------------------------------------------------------------------------------------------------------------------------------------------
%                                                                                     New Section
%----------------------------------------------------------------------------------------------------------------------------------------------------------

\section{Examples}\label{section: examples}

The spherical pairs $(\SO_{9},\Spin_{7})$ and $(H\times H,H)$ with $H=\SL_{n+1}$ and $(\Sp_{2m}\times\Sp_{2n},\Sp_{2m-2}\times\Sp_{2}\times\Sp_{2n-2})$ all admit multiplicity free induction. The first example occurs in Table \ref{table1}, the third in Table \ref{table2} and the second is symmetric; it is the only group-like symmetric pair that admits non-trivial multiplicity free induction (i.e.~other than inducing a one-dimensional representation), see e.g.~\cite[Cor.~4.9]{He et al}. In this section we calculate the spectra $P^{+}_{G}(\mu)$ for the MFSs associated to these spherical pairs. It turns out that the spectra possess a partial ordering that allows for a definition of orthogonal polynomials, see Section \ref{section: OP}.

%----------------------------------------------------------------------------------------------------------------------------------------------------------
%                                                                                     New Subsection
%----------------------------------------------------------------------------------------------------------------------------------------------------------

\subsection{The case $(\SO_{9},\Spin_{7})$ }

Let $G=\SO(9)$, $H\cong\Spin(7)$ and let the embedding $H\subset G$ be given by $\Spin_{7}\subset \SO_{8}\subset\SO_{9}$. The restriction to $\SO_{8}$ of an irreducible $G$-type of highest weight $\lambda=a_{1}\eps_{1}+a_{2}\eps_{2}+a_{3}\eps_{3}+a_{4}\eps_{4}\in P^{+}_{G}$ decomposes into irreducible $\SO_{8}$-modules of highest weight $\nu=b_{1}\eps_{1}+b_{2}\eps_{2}+b_{3}\eps_{3}+b_{4}\eps_{4}$ with
\begin{eqnarray}\label{ineq: B4D4}
a_{1}\ge b_{1}\ge a_{2}\ge b_{2}\ge a_{3}\ge b_{3}\ge a_{4}\ge |b_{4}|.
\end{eqnarray}
Branching from $\SO_{8}$ to $\Spin_{7}$ goes as follows. Let $\tau$ be the outer automorphism of $\SO_{8}$ that interchanges the roots $\alpha_{1}$ and $\alpha_{3}$. Then the highest weights of the irreducible $\Spin_{7}$-representation that occur in the restriction of the irreducible $\SO_{8}$-representation of highest weight $\nu$ are those that occur in $\tau(\nu)$ for the standard embedding $\laso_{7}\subset\laso_{8}$. The same branching rules are obtained if $\tau$ is replaced by any other automorphism that interchanges only $\alpha_{1}\leftrightarrow\alpha_{3}$ or $\alpha_{1}\leftrightarrow\alpha_{4}$.

Restricting the irreducible $\SO_{8}$ representation of highest weight $\nu$ to $\Spin_{7}$ contains a summand of highest weight $\mu=k\eps_{1}$ if and only if  
\begin{eqnarray}
b_{1}+b_{3}\ge k\ge b_{1}+b_{4}\ge0, \label{ineq: D4B3 twist1}\\
b_{2}=b_{1}+b_{3}+b_{4},\label{ineq: D4B3 twist2}\\
b_{2}+b_{3}+b_{4}\le b_{1}.\label{ineq: D4B3 twist3}
\end{eqnarray} 
Indeed, in the basis $\{\eps_{1},\eps_{2},\eps_{3},\eps_{4}\}$ the automorphism $\tau$ is given by the matrix
$$\frac{1}{2}\left(\begin{array}{rrrr}
1&1&1&-1\\
1&1&-1&1\\
1&-1&1&1\\
-1&1&1&1\\
\end{array}\right)$$
and the inequalies follow readily from the classical branching rules. The inequalities (\ref{ineq: D4B3 twist1},\ref{ineq: D4B3 twist2},\ref{ineq: D4B3 twist3}) with $b_{3}\ge|b_{4}|$ imply that $b_{4}=-b_{3}$ and $b_{2}=b_{1}$. Together with (\ref{ineq: B4D4}) one sees that the $\SO_{8}$-modules in the restriction of an irreducible $G$-module of highest weight $\lambda$ are of highest weight $(a_2,a_2,a_4,-a_{4})$. The irreducible $\Spin_{7}$-representation of highest weight $k\eps_{1}$ occurs as a summand in the restriction of the $\SO_{8}$-representation of highest weight $\tau(a_2,a_2,a_4,-a_{4})=(a_{2}+a_{4},a_{2}-a_{4},0,0)$ if and only if $a_{2}+a_{4}\ge k\ge a_{2}-a_{4}$.

\begin{theorem}\label{thm: well}
Let $\mu=k\eps_{1}$. Then $P^{+}_{G}(\mu)=\{a_{1}\eps_{1}+a_{2}\eps_{2}+a_{3}\eps_{3}+a_{4}\eps_{4}\in P^{+}_{G}:a_{2}+a_{4}\ge k\ge a_{2}-a_{4}\}$. Define $B(\mu)=\{s(\varpi_{2}-\varpi_{4})+t(\varpi_{3}-\varpi_{4})+k\varpi_{4}:s,t\in\bbN, s+t\le k\}$. Then the map
$$\lambda:\bbN\varpi_{1}+\bbN\varpi_{4}+B(\mu)\to P^{+}_{G}(\mu)$$
is an isomorphism of sets.
\end{theorem}

\begin{proof}
The description of $P^{+}_{G}(\mu)$ is clear from the discussion above. The fundamental spherical weights for $(G,K)$ are $\varpi_{1}$ and $\varpi_{4}$. It is clear that $\lambda\in P^{+}_{G}(\mu)$ implies $\lambda+u\varpi_{1}+v\varpi_{4}\in P^{+}_{G}(\mu)$ for $u,v\in\bbN$. For $\lambda\in P^{+}_{G}(\mu)$, define $b(\lambda)=\lambda-(a_{1}-a_{2})\varpi_{1}-(a_{2}+a_{4}-k)\varpi_{4}=(a_{2}-a_{3})\varpi_{2}+(a_{3}-a_{4})\varpi_{3}+(a_{4}-a_{2}+k)\varpi_{4}$. Then $b(\lambda)\in P^{+}_{G}(\mu)$ and for $\sigma\in P^{+}_{G}(0)$, $b(\lambda)-\sigma\not\in P^{+}_{G}(\mu)$. Write $s=a_{2}-a_{3}$ and $t=a_{3}-a_{4}$. Then $s,t\in\bbN$ and $s+t=a_{2}-a_{4}\le k$. Hence the map $b:P^{+}_{G}(\mu)\to B(\mu)$ is surjective. It follows that $P^{+}_{G}(\mu)\to \bbN\varpi_{1}+\bbN\varpi_{4}+B(\mu):\lambda\to b(\lambda)+(a_{1}-a_{2})\varpi_{1}+(a_{2}+a_{4}-k)\varpi_{4}$ is an isomorphism.
\end{proof}

\begin{definition}\label{def: degree B4B3}
Define the degree function $|\cdot|:P^{+}_{G}(\mu)\to\bbN$ by $|\lambda+u\varpi_{1}+v\varpi_{4}|=|\lambda|+u+v$ for $u,v\in\bbN$ and $\min(|\lambda+\bbZ\varpi_{1}+\bbZ\varpi_{4}\cap P^{+}_{G}(\mu)|)=0$.
\end{definition}

The bottom is given by $B(\mu)=\{\lambda\in P^{+}_{G}(\mu):|\lambda|=0\}$. We introduce the partial $\preceq_{\mu}$ ordering on $P^{+}_{G}(\mu)$:
$$\lambda'\preceq_{\mu}\lambda\Leftrightarrow|\lambda'|<|\lambda|\mbox{ or }|\lambda'|=|\lambda|\mbox{ and }s'+t'\le s+t,$$
where $\lambda=\lambda(n;s,t)$ and $\lambda'=\lambda(n';s',t')$. One checks that for a weight $\xi_{i}$ of $\pi_{\varpi_{i}}^{G}$, $\lambda+\xi_{i}\preceq_{\mu}+\varpi_{i}$, where $i=1,4$.

%----------------------------------------------------------------------------------------------------------------------------------------------------------
%                                                                                     New Subsection
%----------------------------------------------------------------------------------------------------------------------------------------------------------

\subsection{The case $(H\times H,\Delta(H))$}

Let $H=\SL_{n+1}$ let $B\subset H$ be the Borel subgroup consisting of upper triangular matrices and let $T\subset B$ be the torus consisting of the diagonal elements. Let $P\subset H$ be the parabolic subgroup corresponding to the set $\{\alpha_{1}\}^{c}$. Let $G=H\times H$ and let $\Delta(H)$ be the diagonal, which we identify with by $H$. Then $(G,H,P)$ is a MFS with $(G,H)$ a symmetric pair with involution $\theta(x,y)=(y,x)$. The maximal anisotropic subgroup $A$ is the image of $T\to G:t\mapsto (t,t^{-1})$ and $Z_{H}(A)=\Delta(T)$ which is identified with $T$. Let $B^{-}$ denote the Borel subgroup opposite to $B$ with $B\cap B^{-}=T$. The Borel subgroup $B\times B^{-}$ determines a notion of positivity for roots and weights of $G$. The fundamental weights of $H$ are denoted by $\omega_{i}$. Those of $G$ are subsequently given by $(\omega_{i},0)$ and $(0,-\omega_{i})$. The spherical weights are $(\omega_{i},-\omega_{i})$. Fix $k\in\bbN$. Then
$$P_{T}^{+}(k\omega_{1})=\left\{\sum_{i=1}^{n}k_{i}\alpha_{i}-k\omega_{n}:k\ge k_{n}\ge\cdots\ge k_{1}\ge0\right\}.$$
To see this, lift $\pi_{k\omega_{1}}^{H}$ to a representation of $\GL_{n+1}$ and write its basis in Gelfand-Cetlin tableaux. These consist of zeros everywhere, except for the item in the first entry of each row. These are the $k_{i}$'s in the description. The weights are easily read from these tableaux and they are the ones given in the definition. For example, the highest weight is given by the tuple $(k,k,\ldots,k)$; indeed, $\sum_{i}\alpha_{i}=\omega_{1}+\omega_{n}$. 
Clearly, $(0,-k\omega_{i})\in P^{+}_{G}(k\omega_{i})$. Define $\widetilde{\alpha_{i}}=(\omega_{i}-\omega_{i+1},\omega_{i}-\omega_{i-1})$ (where $\omega_{0}=\omega_{n+1}=0$) and
$$B(k\omega_{1})=\left\{\sum_{i=1}^{n}k_{i}\widetilde{\alpha_{i}}+(0,-k\omega_{n}):k\ge k_{n}\ge\cdots\ge k_{1}\ge0\right\}.$$

\begin{proposition}\label{prop: wellHH}
$P^{+}_{G}(k\omega_{1})=P^{+}_{G}(0)+B(k\omega_{1})$.
\end{proposition}

\begin{proof}
Write $\lambda=k\omega_{i}$. First note that $B(\lambda)$ projects onto $P_{T}^{+}(\lambda)$ after the identification $\omega_{i}\leftrightarrow\frac{1}{2}(\omega_{i},\omega_{i})$. It suffices to show (1) $B(\lambda)\subset P^{+}_{G}(\lambda)$ and (2) $\tau\in B(\lambda), \sigma\in P^{+}_{G}(0)\backslash\{0\}$ implies $\tau-\sigma\not\in P^{+}_{G}(\lambda)$.

(1) An element $\tau\in B(\lambda)$ corresponds to an element $(id\times(-w_{0})^{*})(\tau)=(\mu,\nu)$ in the Weyl chamber that is determined by the Borel subgroup $B\times B$ of $G$ ($w_{0}$ is the longest Weyl group element). The problem is now reduced to the question whether tensor product decomposition of $\pi^{H}_{\mu}\otimes\pi^{H}_{\nu}$ contains $\pi^{H}_{\lambda}$. The element $\tau$ is determined by $(k,k_{n},\ldots,k_{1})$. Upon identifying $\mu$ and $\nu$ with their Young tableaux, one has
$$\mu=(k_{n},k_{n}-k_{1},\ldots,k_{n}-k_{n-1},0),\quad\nu=(k-k_{1},k_{n}-k_{1},\ldots,k_{2}-k_{1},0).$$
Identify $\lambda$ with the Young tableau $\lambda=(k_{n}-k_{1}+k,k_{n}-k_{1},\ldots,k_{n}-k_{1})$. Then $|\mu|+|\nu|=|\lambda|$ and $\mu\subset\lambda$. The skew tableau $\lambda/\mu$ can be filled with natural numbers as follows: The number of columns of $\lambda/\mu$ is $k_{n}-2k_{1}+k$. The top boxes of the first $k-k_{1}$ columns get a 1. In the rows 3 to $n+1$ the number of empty boxes is given by $(k_{2}-k_{1},k_{3}-k_{1},\ldots,k_{n}-k_{1})$. These are precisely the entries $2,\ldots,n$ of $\nu$ in reversed order. Fill the top boxes of the $k_{n}-k_{1}$ columns with 2s. There are $k_{n}-k_{2}$ columns left with empty top box -- fill it with 3s. Continuing in this way, the skew tableau $\lambda/\mu$ gets filled with $k-k_{1}$ 1s, $k_{n}-k_{1}$ 2s and so forth, and on top of every box containing a $3,4,\ldots,n$ there is box containing a $2,3,\ldots,n-1$. This implies that $\lambda/\mu$ is of weight $\nu$ and the row word is a reversed lattice word. This proves (1).

For (2), note that subtracting $\sigma$ amounts to removing $r$ columns in the Young tableaux of $\mu$ and $\nu$, giving $\mu',\nu'$. If the tensor product $\pi^{H}_{\mu'}\otimes\pi^{H}_{\nu'}$ contains $\pi^{H}_{\lambda}$ upon restriction, then the Young tableau of $\lambda$ must be adapted accordingly into $\lambda'$, by removing $r$ columns of $n+1$ boxes. The new number of 1s is $\nu'_{1}=k-k_{1}-r$, whereas the first row of $\lambda'/\mu'$ has $k-k_{1}$ boxes. Hence the skew tableau $\lambda'/\mu'$ of weight $\nu'$ cannot give a reverse lattice word.
\end{proof}

As in the previous example, the spectrum $P^{+}_{G}(k\omega_{1})$ is equipped with a degree function $|\cdot|:P^{+}_{G}(k\omega_{i})\to\bbN$ that counts the number of steps from the bottom $B(k\omega_{i})$ to the given point, see Definition \ref{def: degree B4B3}. In the present example there are $n$ fundamental spherical weights which we denote by $\sigma_{i}=(\omega_{i},-\omega_{i})$.

Let $n=2$ so $(G,H)$ is of rank two, and fix $\mu=k\omega_{1}$. Let $\lambda:\bbN\sigma_{1}+\bbN\sigma_{2}+B(\mu)\to P^{+}_{G}(\mu)$ be the isomorphism of Proposition \ref{prop: wellHH}. We introduce the partial $\preceq_{\mu}$ ordering on $P^{+}_{G}(\mu)$:
$$\lambda'\preceq_{\mu}\lambda\Leftrightarrow|\lambda'|<|\lambda|\mbox{ or }|\lambda'|=|\lambda|\mbox{ and }s'+t'\le s+t,$$
where $\lambda=\lambda(n;s,t)$ and $\lambda'=\lambda(n';s',t')$. One checks that for a weight $\xi_{i}$ of $\pi_{\sigma_{i}}^{G}$, $\lambda+\xi_{i}\preceq_{\mu}+\sigma_{i}$, where $i=1,2$. Similar calculations for $n>2$ soon become cumbersome. It is clear that a conceptual understanding of these phenomena is desired.

%----------------------------------------------------------------------------------------------------------------------------------------------------------
%                                                                                     New Subsection
%----------------------------------------------------------------------------------------------------------------------------------------------------------

\subsection{The case $(\Sp_{2m}\times\Sp_{2n},\Sp_{2m-2}\times\Sp_{2}\times\Sp_{2n-2})$}

Let $F=\Sp_{2m-2}\times\Sp_{2}\times\Sp_{2}\times\Sp_{2n-2}$. The trivial representation of $H$ occurs in a restricted $G$-representation $\pi^{G}_{\lambda}$ if and only if it occurs in the restriction to $H$ of one of the components of the decomposition $\pi^{G}_{\lambda}|_{F}$. Inducing in stages via $F$ and using the inverted branching rules for the cases $(\Sp_{2}\times\Sp_{2},\Sp_{2})$ and $(\Sp_{2m},\Sp_{2m-2}\times\Sp_{2})$ (see Remark \ref{rk: Sp bottom}) yields
$$P^{+}_{G}(0)=\bbN(\varpi_{1},\varpi_{1}')\oplus\bbN(\varpi_{2},0)\oplus\bbN(0,\varpi_{2}'),$$
where $(\varpi_{i},0),(0,\varpi_{i}')$ denote the fundamental weights for $G$. Restricting the $H$-representation of highest weight $\mu=(0,\ell\omega,0)$ to $H_{*}$ is really the restriction of an irreducible $\Sp_{2}$-representation to its maximal torus and it decomposes into $\ell+1$ one-dimensional weight spaces. Certainly, $\{(\ell_{1}\varpi_{1},\ell_{2}\varpi_{1}'):\ell_{1}+\ell_{2}=\ell+2\bbN\}\subset P^{+}_{G}(\mu)$ and in fact, $B(\mu)=\{(\ell_{1}\varpi_{1},\ell_{2}\varpi_{1}'):\ell_{1}+\ell_{2}=\ell\}$. 
Indeed, subtracting $(\varpi_{2},0)$ or $(0,\varpi_{2})$ gives an element outside $P^{+}_{G}$. Subtracting $r(\varpi_{1},\varpi_{1}')$ yields $(p_{1}\varpi_{1},p_{2}\varpi_{1}')$ with $p_{1}+p_{2}=\ell-2r$, whence $(p_{1}\varpi_{1},p_{2}\varpi_{1}')\not\in P^{+}_{G}(\mu)$. It follows that there is an isomorphism $\lambda:P^{+}_{G}(\mu)=P^{+}_{G}(0)+B(\mu)$. The $\mu$-well has dimension $3$ or $4$.

We introduce the partial $\preceq_{\mu}$ ordering on $P^{+}_{G}(\mu)$:
$$\lambda'\preceq_{\mu}\lambda\Leftrightarrow|\lambda'|<|\lambda|\mbox{ or }|\lambda'|=|\lambda|\mbox{ and }n_{1}'\le n_{1},$$
where $\lambda=\lambda(n_{1},n_{2},n_{3};s)$ and $\lambda'=\lambda(n_{1},n_{2},n_{3};s')$. This partial ordering is of a different nature a the those of the previous examples. Nonetheless, one checks that for a weight $\xi_{i}$ of $\pi_{\sigma_{i}}^{G}$, $\lambda+\xi_{i}\preceq_{\mu}\lambda+\sigma_{i}$, where $i=1,2,3$, where
$$\sigma_{1}=(\varpi_{1},\varpi_{1}'),\sigma_{2}=(\varpi_{2},0)\mbox{ and }\sigma_{3}=(0,\varpi_{2}').$$

\begin{remark}
In these three cases the bottom $B(\mu)$ depends affine linearly on $P^{+}_{H_{*}}$. The bottoms of the MFSs of rank one are piece-wise affine linear sets \cite{Heckman van Pruijssen}. In either case, $B(\mu)$ is the translate of a piecewise linear set $B(P)$ that depends only on the parabolic subgroup $P$. The weights of the fundamental spherical representations lie in $B(P)$ translated over the fundamental spherical weights. It would be interesting to understand this in the generality of our classification, perhaps using convexity theorems from symplectic geometry. 

Related to the symplectic point of view, it is interesting whether all the multiplicity free branchings from $H$ to $H_{*}$ can be described as lattice points of a convex polytope. The branching of $\G_{2}$ to $\SO_{4}$ does not have this property: lattice points are missing on the boundary of the convex polytope that contains all $\SO_{4}$-types in the restriction of an irreducible $\G_{2}$-module of highest weight $k\varpi_{2}$ (the shorter fundamental weight). However, this is not a counterexample, for this restriction is not multiplicity free.
\end{remark}

%----------------------------------------------------------------------------------------------------------------------------------------------------------
%                                                                                     New Section
%----------------------------------------------------------------------------------------------------------------------------------------------------------

\section{Orthogonal polynomials}\label{section: OP}

Let $(G,H,P)$ be a MFS from Section \ref{section: examples} or with $(G,H)$ spherical pair of rank one. In the latter case there is a theory that provides families of matrix valued polynomials with nice properties: they are orthogonal, they satisfy a three term recurrence relation and they are simultaneous eigenfunctions of a commutative algebra of differential operators, see \cite{Heckman van Pruijssen, van Pruijssen Roman}.

Let $G_{0}$ and $H_{0}$ denote the compact Lie groups whose Lie algebras are compact forms of $\lag,\lah$. Let $\mu\in P^{+}_{H}$ be the weight of a character of $P$ and consider the space of spherical functions of type $\mu$,
$$E^{\mu}=(R(G_{0})\otimes\End(V^{H}_{\mu}))^{H_{0}\times H_{0}},$$
where $R(G_{0})$ is the convolution algebra of matrix coefficients on $G_{0}$.  Let $\lambda\in P_{G}^{+}(\mu)$ and let $\pi^{G_{0}}_{\lambda}:G_{0}\to\GL(V_{\lambda})$ denote the corresponding representation (the upper index that indicates the group is omitted from now on). Then $V_{\lambda}=V_{\mu}\oplus V_{\mu}^{\perp}$ and we denote by $b:V_{\mu}\to V_{\lambda}$ a unitary $H_{0}$-equivariant embedding and by $b^{*}:V_{\lambda}\to V_{\mu}$ its Hermitian adjoint. The elementary spherical function of type $\mu$ associated to $\lambda\in P_{G}^{+}(\mu)$ is defined by
\[\Phi^{\mu}_{\lambda}:G\to\End(V_{\mu}):g\mapsto b^{*}\circ\pi_{\lambda}(g)\circ b\] 
and it is contained in $E^{\mu}$. The space $E^{\mu}$ is equipped with a sesqui-linear form that is linear in the second variable,
\[\langle\Phi_{1},\Phi_{2}\rangle_{\mu,G}=\int_{G}\tr\left(\Phi_{1}(g)^{*}\Phi_{2}(g)\right)dg\]
with $dg$ the normalized Haar measure on $G_{0}$. Schur orthogonality and the Peter--Weyl Theorem imply: 
\begin{itemize}
\item The pairing $\langle\cdot,\cdot\rangle_{\mu,G}:E^{\mu}\times E^{\mu}\to\bbC$ is a Hermitian inner product and
$$\langle\Phi_{\lambda}^{\mu},\Phi_{\lambda'}^{\mu}\rangle_{\mu,G}=c_{\lambda}\delta_{\lambda,\lambda'},\quad c_{\lambda}=\dim(\mu)^{2}/\dim(\lambda).$$
\item $\{\Phi^{\mu}_{\lambda}:\lambda\in P^{+}_{G}(\mu)\}$ is an orthogonal basis of $E^{\mu}$.
\end{itemize}
Let $U(\lag)$ be the universal enveloping algebra of $\lag$ and let $U(\lag)^{H}$ denote the algebra of differential operators that are invariant under the pull back of right $H$-multiplication. Let $I(\mu)\subset U(\lah)$ be the kernel of the representation $U(\lah)\to\End(V_{\mu})$ and define
$$\bbD(\mu):=U(\lag)^{H}/\left(U(\lag)^{H}\cap U(\lag)I(\mu)\right).$$
The irreducible representations of $\bbD(\mu)$ correspond to irreducible representations of $\lag$ whose restriction to $\lah$ have a subrepresentation of highest weight $\mu$, see e.g.~\cite[Th{\'e}or{\`e}me 9.1.12]{Dixmier}. The algebra $\bbD(\mu)$ is commutative, because all its finite dimensional representations are one-dimensional. Moreover, the elementary spherical functions are simultaneous eigenfunctions for the algebra $\bbD(\mu)$.
 
Equip $P^{+}_{G}(\mu)=P^{+}_{G}(0)+B(\mu)$ with the partial ordering $\preceq_{\mu}$ as defined in Section \ref{section: examples} (for the three examples) or \cite{Heckman van Pruijssen} (for the rank one cases). Denote the string of fundamental zonal spherical functions by $\phi=(\phi_{1},\ldots,\phi_{r})$ and let $\sigma_{i}$ denote the corresponding fundamental spherical weights, i.e.~the $\sigma_{i}$ generate $P^{+}_{G}(0)$. Then $E^{0}=\bbC[\phi]$, i.e.~$E^{0}$ is a polynomial ring\footnote{It should be noted that this may not always be the case, as not all 0-wells are free.}. The product $\phi_{i}\Phi^{\mu}_{\lambda}$ can be expanded in elementary spherical functions of type $\mu$,
\begin{eqnarray}\label{eqn:recurrence el sf}
\phi_{i}\Phi_{\lambda}^{\mu}=\sum_{\lambda-\sigma_{i}\preceq_{\mu}\lambda'\preceq_{\mu}\lambda+\sigma_{i}}c_{\lambda,\lambda'}^{\mu}\Phi^{\mu}_{\lambda'}.
\end{eqnarray}
The fact that the sum may be taken over the indicated set follows the discussion in Section \ref{section: examples} (for the three examples) or \cite{Heckman van Pruijssen} (for the rank one cases).

Define the isomorphism $\lambda:\bbN^{r}\times B(\mu)\to P^{+}_{G}(\mu)$ by $\lambda(d,b)=\sum d_{i}\sigma_{i}+b$, where $d=(d_{1},\ldots,d_{r})$. The Borel-Weil Theorem implies that $c^{\mu}_{\lambda,\lambda+\sigma_{i}}\ne0$. It follows that the elementary spherical function $\Phi^{\mu}_{\lambda}$ can be expressed as a $E^0$-linear combination of the functions $\Phi^{\mu}_{b}$, with $b\in B(\mu)$.

\begin{definition}
\begin{itemize}
\item For $\lambda=\lambda(d,\nu')\in P_{G}^{+}(\mu)$ define $Q_{\lambda}(\phi)=(q^{\mu}_{\nu,\nu'}(\phi):\nu\in B(\mu))$ in $\bbC^{|B(\mu)|}[\phi]$ by
$$\Phi_{\lambda(d,\nu')}^{\mu}=\sum_{\nu\in B(\mu)}q^{\mu}_{\nu,\nu'}(\phi)\Phi_{\lambda(0,\nu)}^{\mu}.$$
\item For every multi-index $d\in\bbN^{r}$ define $Q_{d}(\phi)\in\End(\bbC^{|B(\mu)|})[\phi]$ as the matrix valued polynomial having the $Q_{\lambda(d,\nu')}(\phi)$ as columns ($\nu'\in B(\mu)$).
\end{itemize}
\end{definition}

\begin{theorem}\label{thm:mvmvop}
The matrix valued polynomial $Q_{d}$ is of total degree $|d|$ and the coefficient of degree $d$ is invertible.
\end{theorem}

\begin{proof}
There exist matrices $A_{i,d'}, B_{i,d'},C_{i,d'}\in\End(\bbC^{|B(\mu)|})$
for $i=1,\ldots,r$ and $d'\in\bbN^{r}$, such that
\begin{eqnarray}\label{eqn: rec coefs}
\phi_{i} Q_d(\phi)=\sum_{|d'|=|d|+1}Q_{d'}(\phi)A_{i,d'}+\sum_{|d'|=|d|}Q_{d'}(\phi)B_{i,d'}+\sum_{|d'|=|d|-1}Q_{d'}(\phi)C_{i,d'}.
\end{eqnarray}
Moreover, $A_{i,d+\delta_{i}}$ is upper triangular and invertible, and the other $A_{i,d'}$ are strictly upper triangular, as follows from (the discussion following) (\ref{eqn:recurrence el sf}). 
\end{proof}

Define $V^{\mu}:G\to\End(\bbC^{|B(\mu)|})$ by $V^{\mu}(g)_{\nu,\nu'}=\tr(\Phi_{\lambda(0,\nu)}(g)^*\Phi_{\lambda(0,\nu')}(g))$. Note that $V^{\mu}$ is $K$-biinvariant, hence it is of the form $V^{\mu}=W^{\mu}(\phi)$ with $W^{\mu}$ an element in $\End(\bbC^{|B(\mu)|})[\phi]$. The pairing
\begin{eqnarray}\label{eqn: pairing}
\End(\bbC^{|B(\mu)|})[\phi]\times\End(\bbC^{|B(\mu)|})[\phi]\to\End(\bbC^{|B(\mu)|}),\nonumber\\
\langle Q,Q'\rangle_{W^{\mu}}=\int_G Q(\phi(g))^*W^{\mu}(\phi)Q'(\phi(g))dg,
\end{eqnarray}
is non-degenerate and gives $\End(\bbC^{|B(\mu)|})[\phi]$ the structure of a right pre-Hilbert module over $\End(\bbC^{|B(\mu)|})$ (see e.g.~\cite{Lance} for definitions). Then $\{Q_{d}:d\in\bbN^{r}\}$ is a family of multi-variable matrix valued orthogonal polynomials in the following sense.

\begin{definition} Let $\bbM$ be a finite dimensional matrix algebra over $\bbC$, denote $x=(x_{1},\ldots,x_{r})$ and let $\langle\cdot,\cdot\rangle$ a non-degenerate $\bbM$-valued inner product that makes $\bbM[x]$ into a right pre-Hilbert module over $\bbM$. Let $\{Q_{d}:d\in\bbN^{r}\}\subset\bbM[x]$ be a family of polynomials such that (1) $Q_{d}$ is of total degree $|d|$, (2) for each multi-degree $d\in\bbN^{r}$, the coefficient of $x^{d}$ in $Q_{d}$ is invertible and (3) $\langle Q_{d},Q_{d'}\rangle=\delta_{d,d'}M_{d}$. Then $\{Q_{d}:d\in\bbN^{r}\}$ is called a family of multi-variable matrix valued orthogonal polynomials (MVMVOP). Such a family is called classical if there exists a differential operator $D:\bbM[x]\to\bbM[x]$ that has the polynomials $Q_{d}$ as simultaneous eigenfunctions.
\end{definition}

Starting with an $\bbM$-valued measurable function $W:\bbR^{r}\to\bbM$ and a subset $\Xi\subset\bbR^{r}$ such that $W(x)>0$  on $\Xi$ almost everywhere and
$$\forall d\in\bbN^{r}:\,\int_{\Xi}W(x)x^{d}dx\in\bbM\,\mbox{ (finite moments)}$$
one can produce a family of MVMVOPs with respect to the pairing
$$\langle Q',Q\rangle_{W}=\int_{\Xi}Q'(x)W(x)Q(x)^{*}$$
by application of the Gram-Schmidt procedure. The existence of examples of classical families of MVMVOPs is a priori not clear. However, existence follows from our construction and the fact that the $Q_{d}$ are simultaneous eigenfunctions for the elements in $\bbD(\mu)$.

 If there is a polar decomposition $G_{0}=H_{0}A_{0}H_{0}$ for the pair $(G_{0},H_{0})$ then we may replace the integration in (\ref{eqn: pairing}) by an integration over $A_{0}$. This is possible if $(G,H)$ is symmetric or spherical of rank one.

Summing up, the indicated MFSs give rise to families of orthogonal polynomials. The families $(Q_{d}:d\in\bbN^{r})$ are so called multi-variable matrix valued orthogonal polynomials and the orthogonality is matrix valued. One could also study the families $(Q_{\lambda(d,\nu)}:d\in\bbN^{r},\nu\in B(\nu))$ of multi-variable vector valued orthogonal polynomials. The matrix weight remains the same but the pairing is now scalar valued.

If there is a polar decomposition $G_{0}=H_{0}A_{0}H_{0}$ for the pair $(G_{0},H_{0})$, then the matrix valued polynomials play a role in the harmonic analysis of sections of (specific) vector bundles over $G_{0}/H_{0}$, just as ordinary Jacobi polynomials (with geometric parameters) do for analyzing $K_{0}$-invariant functions on a symmetric space $G_{0}/K_{0}$. Moreover, after taking radial parts of the differential operators in $\bbD(\mu)$ and subsequently change the variables $x=\phi$, we obtain an interpretation of the algebra $\bbD(\mu)$ as an algebra differential operators whose coefficients are polynomials. This could shine a new light on the understanding of the algebra $\bbD(\mu)$, see also \cite[Conj.~10.3]{Lepowsky}.

%----------------------------------------------------------------------------------------------------------------------------------------------------------
%                                                                                     New Section
%----------------------------------------------------------------------------------------------------------------------------------------------------------

\section{Outlook}\label{section: structure}

Our aim is to reduce the number of variables of the families of polynomials and in particular, to make the pairing $\langle\cdot,\cdot\rangle_{W}$ more explicit, as an integration over an $r$-dimensional compact domain. We have already seen that the polynomials are really polynomials in the fundamental spherical functions $\phi_{i}$ and that this also holds for the matrix weight. This implies that the functions $Q_{d}$ and $W^{\mu}$ are $H_{0}$-biinvariant. If the spherical pair is symmetric, then there is a polar decomposition $G_{0}=H_{0}A_{0}H_{0}$ and this also holds for the spherical pairs of rank one. We show that there is a polar decomposition on the level of the algebraic groups and we indicate the difficulty of passing to a polar decomposition for the compact real forms.

The local structure theory implies that $G/S$ ($S$ is the horospherical type of $H\subset G$, see also Section \ref{section: spectrum}) admits an open orbit for the action $H\times A$. We can also understand this using representation theory.

\begin{lemma}
The group $H\times A$ acts on $X_{0}=G/S$ via left and right multiplication which we denote by $\phi:H\times A\to X_{0},\phi_{0}(h,a)=haS/S$. This action admits an open orbit.
\end{lemma}

\begin{proof}
Let $m\in k[X_{0}]$ and suppose that $m(ha)=0$ for all $h,a$. Write $m=\sum_{i}m_{f_{i},v_{i,0}}$ with $v_{i,0}$ a highest weight vector of $A$ in $V^{G}_{\lambda_{i}}$. Then $m_{f_{i},v_{i,0}}(ha)=0$, otherwise we would find two characters of $A$ that were dependent. Hence $f_{i}(hv_{i,0})=0$ for all $h\in H$. As $\langle Hv_{i,0}\rangle=V_{\lambda}^{G}$, $f_{i}=0$. This shows $m=0$ and thus $\phi^{*}(m)=0$ implies $m=0$, whence $\phi^{*}$ injective and $\phi$ is dominant.
\end{proof}

Before we adapt this proof to show that the map $H\times A\to X=G/H$ is dominant, we prove a simple result.

\begin{lemma}
Let $V^{G}_{\lambda}$ be the representation space of a spherical weight. Let $v^{H}\ne0$ be an $H$-invariant vector. Let $v_{0}$ be an $A$-weight vector of weight $\lambda$. The coefficient of $v_{0}$ in the decomposition of $v^{H}$ in $A$-weight vectors is non-zero.
\end{lemma}

\begin{proof}
Let $\{v_{0},\ldots,v_{d}\}$ be a basis of $A$-weight vectors, with $v_{0}$ of $A$-weight $\lambda$ and let $\{f_{0},\ldots,f_{d}\}$ be its dual basis. Let $v^{H}\ne0$ be an $H$-fixed vector in $V_{\lambda}^{G}$. Then $f_{0}(v^{H})\ne0$. Indeed, if not, then $(hf_{0})(v^{H})=0$ for all $h\in H$. This would imply $v^{H}=0$, since $\langle Hf_{0}\rangle=(V_{\lambda}^{G})^{\vee}$, which is absurd.
\end{proof}

\begin{theorem}\label{thm: G=HAH}
The map $\psi:H\times A\to X=G/H:(h,a)\mapsto haH/H$ is dominant.
\end{theorem}

\begin{proof}
The map $\psi$ is well defined: $A=L/L_{0}$ and $L_{0}=H_{*}\subset H$. Let $m\in\bbC[X]$ be zero on the image of $\psi$.  Write $m=\sum m_{f_{i},v_{i}^{H}}$, where $v_{i}^{H}\in V^{G}_{\lambda_{i}}$ is $H$-invariant. Write $v^{H}_{i}=\sum v_{i,j}$ in $A$-weights, with $v_{i,0}$ of weight $\lambda_{i}$. Furthermore, write $wts(i)$ for the set of $A$-weights $\sigma$ such that a non-zero weight vector of weight $\sigma$ occurs in the decomposition of $v^{H}_{i}$ into $A$-weight vectors. Let $wts(m)$ be the union of all these sets for which $m_{f_{i},v_{i}^{H}}\ne0$. This is a set with multiplicities. Let $\lambda_{i}\in wts(m)$ be an element with multiplicity one. Such a weight exists because $\bbC[X]$ decomposes multiplicity free as $G$-module.

The assumption $m(ha)=0$ yields a linear expression of characters of $A$ which is equal to zero. The coefficient of $\lambda_{i}$ is $f_{i}(hv_{i,0})$ which must thus be zero for all $h\in H$. Since $\langle Hv_{i,0}\rangle=V^{G}_{\lambda_{i}}$, it follows that $f_{i}=0$. An induction argument implies that $wts(m)=\emptyset$. In particular, $m=0$, as was to be shown.
\end{proof} 
 
For the compact real forms we obtain a map $H_{0}\times A_{0}\to G_{0}/H_{0}$ whose image $Y=\psi_{0}(H_{0}\times A_{0})$ is a compact neighborhood of $eH_{0}$. The question is whether $\psi_{0}$ is surjective. A possible approach is to show that a matrix coefficient that is arbitrarily small on $Y$ cannot have value $1$ in a given point $x\in Y^{c}$.
 
In the case $\Spin(9)/\Spin(7)=S^{15}$ we can parametrize the $\Spin(7)$-orbits by points of a sphere $S^{2}\subset S^{15}$. Indeed, $S^{15}\subset\bbR^{16}$ is the $\Spin(9)$-orbit of the first basis vector $e_{1}$. The next 8 basis vectors $(e_{2},\ldots,e_{9})$ span the representation space for the $\Spin(7)$-action of highest weight $\omega_{3}$ and the last 7 basis vectors $(e_{10},\ldots,e_{16})$ that of highest weight $\omega_{1}$. Given a point $(re_{1},v,w)\in\bbR^{1}\oplus\bbR^{8}\oplus\bbR^{7}$ with $r^2+||v||^{2}+||w||^{2}=1$, we can translate it to $re_{1}+||v||e_{2}+||w||e_{10}$ using $\Spin(7)$. First we bring the second coordinate $v$ in position $||v||e_{2}$ using the fact that $S^{7}$ is homogeneous for the $\Spin(7)$-action. The isotropy group for this action is $\G_{2}$ which acts transitively on $S^{6}$ and so we rotate $w$ to $||w||e_{10}$. 

However, this is not the decomposition we are aiming for. We hope to find a general approach understand polar decompositions for spherical pairs. Possibly the techniques as the ones used in \cite{Kobayashi3, Tanaka} may be useful. However, as far as we know, the polar decomposition $G_{0}=H_{0}A_{0}H_{0}$ for compact spherical pairs remains an open problem.\\ \\

\textbf{Acknowledgment}. I am grateful to \textsc{Michel Brion, Dennis Brokemper} and \textsc{Pablo Rom{\'a}n} for their comments on earlier versions of this manuscript.


\begin{thebibliography}{20}




\bibitem{Baldoni-Silva}
M.W. Baldoni Silva, Branching theorems for semisimple Lie groups of real rank one,
Rend. Sem. Univ. Padova 61 (1979), 229--250.                                                                                           


\bibitem{Benson and Ratcliff}
C.~Benson, G.~Ratcliff,
A classification of multiplicity free actions,
J. Algebra 181 (1996), no. 1, 152--186. 


\bibitem{Brion2}
M.~Brion,
Représentations exceptionnelles des groupes semi-simples,
Ann. Sci. École Norm. Sup. (4) 18 (1985), no. 2, 345--387.


\bibitem{Brion1}
M.~Brion, Classification des espaces homog\`{e}nes sph\'{e}riques, Compositio Math. 63:2 (1987), 189--208.

\bibitem{Camporesi}
R.~Camporesi, The Helgason Fourier transform for homogeneous vector bundles over compact Riemannian symmetric spaces--the local theory,
J. of Funct. Analysis 220 (2005), 97--117.


\bibitem{Dixmier}
J.~Dixmier, Alg{\`e}bres envelloppantes, {\'E}ditions Jacques Gabay, Paris, 1996.


%\bibitem{Gortz and Wedhorn}
%U.~G{\"o}rtz, T.~Wedhorn,
%Algebraic geometry I,
%Schemes with examples and exercises, Vieweg + Teubner, Wiesbaden, 2010.


\bibitem{GPT}
F.A.~Gr\"{u}nbaum, I.~ Pacharoni, J.A.~Tirao,
Matrix valued spherical functions associated to the complex projective plane, J. Funct. Anal. 188 (2002), 350-441.


\bibitem{He et al}
X.~He, H.~Ochiai, K.~Nishiyama, Y.~Oshima,
On orbits in double flag varieties for symmetric pairs, Transform. Groups 18 (2013), no. 4, 1091--1136. 


\bibitem{Heckman van Pruijssen}
G.~Heckman, M.~van Pruijssen, Matrix valued orthogonal polynomials for Gelfand pairs of rank one, arXiv:1310.5134.


\bibitem{Kac}
V.G.~Kac,
Some remarks on nilpotent orbits,
J. Algebra 64 (1980), no. 1, 190--213. 


\bibitem{Kitagawa}
M.~Kitagawa,
Stability of branching laws for spherical varieties and highest weight modules.
Proc. Japan Acad. Ser. A Math. Sci. 89 (2013), no. 10, 144--149.


\bibitem{KnappLGBI}
A.W.~Knapp, Lie Groups Beyond an Introduction, Second Edition,
Progress in Mathematics 140, Birkh\"auser, Boston, 2002.


\bibitem{Knop: some remarks}
F.~Knop,
Some remarks on multiplicity free spaces, Representation theories and algebraic geometry (Montreal, PQ, 1997), 301–317,
NATO Adv. Sci. Inst. Ser. C Math. Phys. Sci., 514, Kluwer Acad. Publ., Dordrecht, 1998.

\bibitem{Knop van Steirteghem}
F.~Knop, B.~Van Steirteghem, Classification of smooth affine spherical varieties, Transformation Groups, Vol. 11, No. 3, (2006),495--516.


\bibitem{Kobayashi1}
T.~Kobayashi,
Multiplicity-free representations and visible actions on complex manifolds,
 Publ. Res. Inst. Math. Sci. 41 (2005), no. 3, 497--549.


\bibitem{Kobayashi2}
T.~Kobayashi,
Visible actions on symmetric spaces,
Transform. Groups 12 (2007), no. 4, 671--694.


\bibitem{Kobayashi3}
T.~Kobayashi,
A generalized Cartan decomposition for the double coset space (U(n1)×U(n2)×U(n3))∖U(n)/(U(p)×U(q)),
J. Math. Soc. Japan 59 (2007), no. 3, 669--691. 


\bibitem{Koelink--van Pruijssen--Roman1}
E.~Koelink, M.~van Pruijssen, P.~Rom\'{a}n,
Matrix Valued Orthogonal Polynomials related to $(\SU(2)\times\SU(2),\mathrm{diag})$, Part I,
Int Math Res Notices (2012) 2012 (24): 5673--5730. 


\bibitem{Koelink--van Pruijssen--Roman2}
E.~Koelink, M.~van Pruijssen, P.~Rom\'{a}n,
Matrix Valued Orthogonal Polynomials associated to $(\SU(2)\times\SU(2),\SU(2))$, Part II,
Publ. RIMS Kyoto Univ. 49 (2013), 271--312.


\bibitem{Koornwinder}
T.H.~Koornwinder, Matrix elements of irreducible representations of $\SU(2)\times\SU(2)$
and vector-valued orthogonal polynomials, SIAM J. Math. Anal. 16 (1985), 602--613.


\bibitem{Kramer}
M.~Kr\"{a}mer, Sph\"{a}rische Untergruppen in kompakten zusammenh\"{a}ngen-den Liegruppen,
Compositio Math. 38 (1979), 129--153.


\bibitem{Lance}
E.C.~Lance, Hilbert C*-Modules, A toolkit for operator algebraists,
Cambridge University Press,1995.


\bibitem{Leahy}
A.S.~Leahy,
A classification of multiplicity free representations,
J. Lie Theory 8 (1998), no. 2, 367--391.

\bibitem{Lepowsky}
J.~Lepowsky,
Algebraic results on representations of semisimple Lie groups,
Trans.Amer.Math.Soc.,176 (1973), 1--44.

\bibitem{Panyushev}
D.I.~Panyushev,
Complexity and rank of homogeneous spaces.
Geom. Didicata 34 (1990), no. 3, 249--269. 


\bibitem{van Pruijssen Roman}
M.~van Pruijssen, P. Rom{\'a}n, Matrix valued classical pairs related to compact Gelfand pairs of rank one, arXiv:1312.6577.

\bibitem{Tanaka}
Y.~Tanaka,
Visible actions on flag varieties of type B and a generalisation of the Cartan decomposition,
Bull. Aust. Math. Soc. 88 (2013), no. 1, 81--97.

\bibitem{Timashev}
 D.A.~Timashev,
Homogeneous spaces and equivariant embeddings, Encyclopaedia of Mathematical Sciences, 138 Springer, Heidelberg, 2011.


\bibitem{Tirao}
J.A.~Tirao, The matrix-valued hypergeometric equation, PNAS 100:14 (2003), 8138--8141.


\bibitem{Wallach}
N.R.~Wallach, Harmonic Analysis on Homogeneous Spaces, Marcel Dekkers, New York, 1973.


\end{thebibliography}
\end{document}